\documentclass[a4paper]{amsart}
\usepackage{amssymb,color}
\usepackage{graphicx}
\graphicspath{{figs/}}
\usepackage{upref}
\usepackage{url}
\definecolor{darkred}{rgb}{0.6,0,0}
\definecolor{darkgreen}{rgb}{0,0.5,0}
\definecolor{darkmagenta}{rgb}{0.5,0,0.5}
\usepackage[colorlinks=true,
  linkcolor=darkred,
  citecolor=darkgreen,
  urlcolor=darkmagenta]{hyperref}

\newcommand{\abs}[1]{\left\vert#1\right\vert}
\newcommand{\dott}{\, \cdot\,}

\newcommand{\indicM}[1]{\mathbf{1}_{#1}}

\newcommand{\R}{\mathbb R}
\newcommand{\N}{\mathbb N}

\newcommand{\Dy}{\Delta y}

\newcommand{\seq}[1]{\left\{#1\right\}}

\newcommand{\norm}[1]{\left\|#1\right\|}
\newcommand{\lipnorm}[1]{\norm{#1}_{\mathrm{Lip}}}
\DeclareMathOperator{\sign}{sign}
\newcommand{\sgn}[1]{\sign\left(#1\right)}
\newcommand{\test}{\varphi}
\newcommand{\DM}{\Delta^-_i}
\newcommand{\DP}{\Delta^+_i}

\newcommand{\Dm}{D^-}
\newcommand{\Dp}{D^+}
\newcommand{\intom}{\int_{\Pi_T}\!\!}
\newcommand{\intomeg}{\int_{\Omega}\!}
\newcommand{\intomega}{\int_{\Omega_T}\!\!}
\newcommand{\baru}{\bar{u}}

\numberwithin{equation}{section}     
\allowdisplaybreaks

\newtheorem{theorem}{Theorem}[section]

\newtheorem{lemma}[theorem]{Lemma}
\newtheorem{definition}[theorem]{Definition}

\newtheorem{remark}[theorem]{Remark}
\newtheorem{corollary}[theorem]{Corollary}
\newtheorem*{remark*}{Remark}

\numberwithin{equation}{section}     
\allowdisplaybreaks

\begin{document}
\title[]{Models for dense multilane vehicular traffic}

\author[Holden]{Helge Holden} \address[Holden]{\newline Department of
  Mathematical Sciences, NTNU Norwegian University of Science and
  Technology, NO--7491 Trondheim, Norway}
\email[]{\href{helge.holden@ntnu.no}{helge.holden@ntnu.no}}
\urladdr{\href{https://www.ntnu.edu/employees/helge.holden}{https://www.ntnu.edu/employees/helge.holden}}

\author[Risebro]{Nils Henrik Risebro} \address[Risebro]{\newline
  Department of Mathematics, University of Oslo, P.O.\ Box 1053,
  Blindern, NO--0316 Oslo, Norway }
\email[]{\href{nilshr@math.uio.no}{nilshr@math.uio.no}}
\urladdr{\href{https://www.mn.uio.no/math/english/people/aca/nilshr/index.html}
{https://www.mn.uio.no/math/english/people/aca/nilshr/index.html}}

\date{\today}



\thanks{Research was supported by the grant {\it Waves and Nonlinear Phenomena (WaNP)} (no 250070) from the Research Council of Norway.}

\begin{abstract}
  We study vehicular traffic on a road with multiple lanes and dense,
  unidirectional traffic following the traditional
  Lighthill--Whitham--Richards model where the velocity in each lane
  depends only on the density in the same lane.  The model assumes
  that the tendency of drivers to change to a neighboring lane is
  proportional to the difference in velocity between the lanes. The
  model allows for an arbitrary number of lanes, each with its
  distinct velocity function.

  The resulting model is a well-posed weakly coupled system of
  hyperbolic conservation laws with a Lipschitz continuous source. We
  show several relevant bounds for solutions of this model that are
  not valid for general weakly coupled systems.
  
  Furthermore, by taking an appropriately scaled limit as the number
  of lanes increases, we derive a model describing a continuum of
  lanes, and show that the $N$-lane model converges to a weak solution
  of the continuum model.
\end{abstract}

\maketitle

\section{Introduction} \label{sec:intro}
The Lighthill--Whitham--Richards (LWR) model for unidirectional
traffic on a single road, see
\cite{LW_II,richards}, reads 
\begin{equation} \label{eq:lwr}
u_t+(u v(u))_x=0, 
\end{equation}
where $u=u(t,x)$ denotes the density of vehicles at the position $x$
and time $t$, and $v=v(u)$ is a given velocity function. The LWR-model
expresses conservation of vehicles and is a well-established
model for dense unidirectional single lane vehicular traffic on a
homogeneous road without exits and entries.  Furthermore, it serves as
the standard textbook example to gain intuition regarding the behavior
of solutions of scalar one-dimensional hyperbolic conservation laws,
see, e.g., \cite{HoldenRisebro}.

Given the importance of vehicular traffic
modeling in modern society, it is no wonder that the LWR-model has
been generalized to describe several important scenarios in 
dense traffic flow. Indeed, ``traffic hydrodynamics'' has become a
research field in its own right, where the flow of vehicles is modeled
by conservation laws or balance equations. In the general context, the
LWR-model is the simplest model among the many hydrodynamic traffic
models. Among the other models often used is the Aw--Rascle model
\cite{AwRascle}, which is a system of conservation laws where
the velocity $v$ is not a given function of $u$, but satisfies a
second conservation law. It is thus considerably more complicated than
the simple LWR-model. 
For a general
introduction to how conservation laws are used in traffic modeling,
see \cite{MR1338371,MR3553143}  
and the many references
therein. 


In this paper we introduce a new model for multilane dense vehicular
traffic where the underlying model for each lane remains the
LWR-model. Our basic assumption is that drivers prefer to drive
faster, and that the tendency of a vehicle to change lane is
proportional to the difference in velocity between neighboring lanes.
If \eqref{eq:lwr} describes the density of vehicles in a particular
lane, the multilane behavior is described by a source term, accounting
for lane changes. The result is thus a system of weakly coupled scalar
conservation laws.

More precisely, consider two lanes denoted $1$ and $2$, the model we
study, reads
\begin{align*}
  \partial_t u_{1}+\partial_x(u_1  v_1(u_1))&=-S(u_1,u_{2}), \\
  \partial_t u_{2}+\partial_x(u_2  v_2(u_2))&=S(u_1,u_{2}),
\end{align*}
where the change of lanes is codified in 
\begin{equation*}
  S(u_1,u_{2}) = K(v_{2}(u_2)-v_1(u_1))\cdot
  \begin{cases}
    u_1 & v_{2}(u_2)\ge v_1(u_1),\\
    u_{2} & v_{2}(u_2)<v_1(u_1).
  \end{cases}
\end{equation*}
Here $u_i$ denotes the density in lane $i$. 
The system constitutes a weakly coupled $2\times 2$ system of  one-dimensional  hyperbolic conservation laws, and there is ample theory available for systems of this type, see 
Section~\ref{sec:cont2}.  The system readily generalizes to an arbitrary number of lanes, see Section \ref{sec:cont2a}.  We show that the general system with $N$ lanes has a unique 
entropy solution, and that the solution is well-posed in the sense that one has a surprising $L^1$ stability
\begin{equation*}
  \sum_{i=1}^N \norm{u_i(t)-\baru_i(t)}_{L^1(\R)} 
  \le   \sum_{i=1}^N
  \norm{u_{i,0}-\baru_{i,0}}_{L^1(\R)},
\end{equation*}
for two solutions $u_i$ and $\baru_i$, see Theorems
\ref{thm:existence} and \ref{thm:L1contr}. Note that the $L^1$
stability \textit{does not} hold in general for systems of balance
laws, that is, hyperbolic conservation laws with source.

The models invites for considering the continuum limit where the number of lanes increases to infinity. We organize the parallel lanes along the $x$-axis, and measure the distance 
between the lanes along the $y$-axis.  The distance between the lanes is scaled as $\Dy=1/N$,  where $N$ denotes the number of lanes. For simplicity we assume that the velocity function is given by 
$v_i(u)=-k(y_i)g(u)$ where $y_i=i\Dy$, and $-g(u)$ is the velocity function. We scale the function such that $g(0)=-1$ and $g(1)=0$. We need to scale the constant $K$ as $K=1/\Dy^2$. 
We consider given initial data $u_0\colon\R\times [0,1] \to [0,1]$, where the initial data for lane $i$ is $u_{i,0}$ is given by \eqref{eq:init_N} and with solution $u_i$.   We interpolate this
function to $u_{\Dy}$ where $u_{\Dy}\colon[0,\infty)\times\R\times [0,1] \to [0,1]$. We assume that $k$ is smooth and positive with  $k'(0)=k'(1)=0$.  In Theorem \ref{thm:infinite_theorem} we show 
that $u_{\Dy}\to u$ where $u$ is a weak solution of 
\begin{equation*}
  \begin{cases}
    u_t + kf(u)_x + (k' f(u))_y = \left(k u g_y\right)_y,\\
    g(u)_y|_{y=0,1} = 0,\\
    u|_{t=0}=u_0,
  \end{cases}
\end{equation*}
where the \emph{flux function} $f$ is defined as $f(u)=uv(u)$.  This
equation is an interesting anisotropic and degenerate parabolic
equation with non-trivial boundary conditions in the $y$-direction.

There is a plethora of approaches to the modeling of multilane dense traffic, and the most
relevant to our approach here can be found in
\cite{HelbingGreiner,MR1671940,MR1671918,MR770381}, using either kinetic models or the
Aw--Rascle model or variations thereof, or \cite{MR2273100,MR1969676} where  more involved
source terms modeling the change of lanes, are employed. See \cite{Moridpour} for a survey of
various models for lane changing.

The rest of this paper is organized as follows: In
Section~\ref{sec:cont2} we detail the two-lane case, and show that
$u_i\in [0,1]$ is an invariant region. In Section~\ref{sec:cont2a} we
state the $N$-lane model, and prove a number of estimates on the
solution. Finally, in Section~\ref{sec:inftylane}, we study the limit
as $N\to\infty$.  See \cite{MR3811561} for a model for two-dimensional  traffic flow on highways. 
Analogously to the analysis of numerical schemes for
degenerate parabolic equations, we establish enough estimates on the
solution, enabling us to conclude that a limit exists, and that this
limit is a weak solution of a degenerate convection-diffusion
equation. All sections are illustrated by numerical examples.

\section{A continuum model for two-lane vehicular traffic} \label{sec:cont2} 

Consider a road with two  lanes, each with its own velocity function. The lanes are
homogeneous,  and traffic on the road is unidirectional. We assume that the vehicular 
traffic is dense,  allowing for a continuum formulation. Let $u_i$ and
$v_i=v_i(u_i)$ denote the density and velocity, respectively, in
lane $i$. 

In this paper we focus on the interaction between the two lanes. We assume that
drivers prefer to drive in the faster lane, and the 
tendency of a vehicle to
change lane is proportional to the difference in velocity. Thus the
flow \emph{from} lane $1$ \emph{to} lane $2$ equals
\begin{align}
  S(u_1,u_{2}) &= K(v_{2}(u_2)-v_1(u_1))\cdot
  \begin{cases}
    u_1 & v_{2}(u_2)\ge v_1(u_1),\\
    u_{2} & v_{2}(u_2)<v_1(u_1),
  \end{cases}\notag\\
  &=K\left[\left(v_2(u_2)-v_1(u_1)\right)^+ u_1 -
    \left(v_2(u_2)-v_1(u_1)\right)^- u_2 \right] \label{eq:diffdef}
\end{align}
where  $K$ is a constant, $(a)^+=\max\seq{a,0}$ and $(a)^-=-\min\seq{a,0}$.  The
flow \emph{from} lane $2$ \emph{to} lane $1$ equals
$-S(u_1,u_{2})$.   The classical Lighthill--Whitham--Richards model 
implies the following model describing the two-lane traffic
\begin{subequations} \label{eq:ligningen}
\begin{align}
  \partial_t u_{1}+\partial_x(u_1
  v_1(u_1))&=-S(u_1,u_{2}),\label{eq:ligningen1a} \\
  \partial_t u_{2}+\partial_x(u_2
  v_2(u_2))&=S(u_1,u_{2}), \label{eq:ligningen1b}
\end{align}
\end{subequations} 
where $x$ is the position along the road and $t$ denotes time.
This $2\times 2$ system of hyperbolic conservation laws is weakly coupled with a
Lipschitz continuous source term. 

The velocities $v_i=v_i(u_i)$ are strictly decreasing positive
functions, and we assume that they are scaled such that
$v_1(1)=v_2(1)=0$. For simplicity, we scale space and time such that $K=1$. 

It is well-known that  this system in general only allows for weak solutions $u_i\in L^1(\R)\cap BV(\R)$, the set of integrable functions of finite total variation, see, e.g.,  \cite{HoldenRisebro}.  Furthermore, the issue of uniqueness of the solution is non-trivial and one needs to require that the solution satisfies an entropy condition. 
\begin{definition} \label{def:entropy2lane}
  Let $v_i=v_i(u_i)$ be strictly decreasing positive
  functions such that $v_1(1)=v_2(1)=0$. Assume that $u_{i,0}\in
  L^1([0,1])\cap BV([0,1])$ for $i=1,2$. We say that $u_{i}\in
  C([0,\infty); L^1(\R))$ with $u_i(t, \dott)\in BV(\R)$ for $t\in
  [0,\infty)$ is a weak solution of \eqref{eq:ligningen} with initial
  data $u_{i,0}$ if
  \begin{align*}
    \int_0^\infty \int_\R\big( u_{1}\test_t+u_1
    v_1(u_1)\test_x-S(u_1,u_2)\test\big)\, dxdt+ \int_\R u_{1,0}\test|_{t=0}\, dx&=0, \\
    \int_0^\infty \int_\R\big( u_{2}\test_t+u_2
    v_2(u_2)\test_x+S(u_1,u_2)\test\big)\, dxdt+ \int_\R
    u_{2,0}\test|_{t=0}\, dx&=0,
  \end{align*}
  for all compactly supported test functions $\test\in
  C^\infty_0([0,\infty)\times \R)$.

  The solution is called an entropy solution if
  \begin{subequations} \label{eq:kruz}
    \begin{align}
      \int_0^\infty \int_\R\big(
      \eta(u_{1})\test_t+q_1(u_1)\test_x\big)\, dxdt+& \int_\R
      \eta(u_{1,0})\test|_{t=0}\, dx  \notag\\ 
      &\qquad\ge\int_0^\infty \int_\R \eta'(u_1)\test S(u_1,u_{2})\, dxdt, \\
      \int_0^\infty \int_\R\big(
      \eta(u_{2})\test_t+q_2(u_2)\test_x\big)\, dxdt+& \int_\R
      \eta(u_{2,0})\test_{t=0}\, dx \notag\\ 
      &\qquad\ge -\int_0^\infty \int_\R\eta'(u_2)\test S(u_1,u_{2})\,
      dxdt,
    \end{align}
  \end{subequations}
  for all convex functions $\eta$ where $q_i$ satisfies
  $q_i'(u)=\eta'(u)f_i'(u)$ with $f_i(u)=u v_i(u)$, and for all
  compactly supported non-negative test functions $\test\in
  C^\infty_0([0,\infty)\times \R)$, $\test\ge0$.
\end{definition}
\begin{remark}
  It suffices that \eqref{eq:kruz} holds for $\eta$ of the form
  $\eta(u)=\abs{u-k}$ for all constants $k\in\R$, see \cite[Remark
  2.1]{HoldenRisebro}. In that case $q_i(u)=\sgn{u-k}(f_i(u)-f_i(k))$.
\end{remark}
\begin{remark}
  The existence and uniqueness of entropy solutions to
  \eqref{eq:ligningen} follows by Theorem~\ref{thm:existence} below.
\end{remark}

We will throughout the paper use the following notation:
\begin{equation}
a^\pm= \frac12\big(\abs{a}\pm a \big), \quad H(a)=\indicM{[0,\infty)}(a),
\end{equation}
where $\indicM{M}$ is the indicator (characteristic) function of a set $M$.
Note that
\begin{align*}
  0&\le a^\pm\le \abs{a}, \quad \abs{a}=a^+ +a^-, \quad a=a^+ -a^-,
  \quad a^+ a^-=0,
  \quad (\mp a)^- = (\pm a)^+,\\
  &H(x)+H(-x)=1, \quad (x^+)'=H(x), \quad (x^-)'=-H(-x), \; x\ne 0.
\end{align*}
We shall also employ the convention that $C$ denotes a ``generic''
finite positive constant, independent of critical parameters, whose
actual value may change from one occurrence to the next. Similarly, we use
$C_\alpha$ to denote a positive function $c(\alpha)<\infty$ for $\alpha<\infty$.

This  model \eqref{eq:ligningen} has the natural invariant region $u\in [0,1]$. This is the content of the following lemma. 
\begin{lemma} \label{lem:invariant}
  Let $u_1$ and $u_2$ be entropy solutions in the sense of 
  Definition~\ref{def:entropy2lane}, with initial data $u_{i,0}$ for
  $i=1,2$. If $u_{i,0}(x)\in [0,1]$ for all $x$ and $i=1,2$, then 
  $u_i(t,x)\in [0,1]$ for all $x$  and for $t>0$.
\end{lemma}
\begin{proof}
  To show that $u_i\ge 0$ if $u_{i,0}\ge 0$ we use the entropy
  $\eta(u)=u^-$. Then
  \begin{equation*}
    \partial_t (u_i)^- + \partial_x q_i^-(u_i) = (-1)^i H(-u_i) S(u_1,u_2)
  \end{equation*}
  in $\mathcal{D}'$ for $i=1,\,2$. We use a non-negative test function
  $\test(x,t)\approx \indicM{[0,\tau]}$ to find that
  \begin{equation*}
    \int_\R (u_i(\tau,x))^- \,dx \le \int_\R (u_{i,0}(x))^- \,dx
    + (-1)^{i+1} 
    \int_0^\tau\int_\R H(-u_i)S(u_1,u_2)\,dxdt,
  \end{equation*}
   Adding these two equations and using that
  $(u_{i,0})^-=0$, we get
  \begin{equation*}
    \int_\R (u_1(\tau,x))^- + (u_2(\tau,x))^- \,dx \le
    \int_0^\tau\int_\R r(u_1,u_2)\,dxdt, 
  \end{equation*}
  with
  \begin{equation*}
    r(u_1,u_2)=S(u_1,u_2)(H(-u_1)-H(-u_2)).
  \end{equation*}
  We have that
  \begin{align*}
    r(u_1,u_2)&=
    \begin{cases}
      0 & u_1 < 0\ \text{and}\ u_2< 0,\\
      0 & u_1 > 0\ \text{and}\ u_2> 0,\\
      -\left[\left(v_2(u_2)-v_1(u_1)\right)^+ u_1 -
    \left(v_2(u_2)-v_1(u_1)\right)^- u_2 \right]  & u_2 \le 0 < u_1,\\
     \left[\left(v_2(u_2)-v_1(u_1)\right)^+ u_1 -
    \left(v_2(u_2)-v_1(u_1)\right)^- u_2 \right]  & u_1 \le 0 < u_2,
    \end{cases}\\
    &\le 0.
  \end{align*}
  Hence $u_i(\tau,x)\ge 0$ for almost all $x$.
  
  Similarly,
  by using the convex entropy $\eta(u)=(u-1)^+$ we get
  \begin{equation*}
    \partial_t (u_i-1)^+ + \partial_x q_i^+(u_i) \le
    (-1)^{i+1}H(u_i-1)S(u_1,u_2)
  \end{equation*}
  in $\mathcal{D}'$, the set of distributions. By the same argument as before, we arrive at
  \begin{equation*}
    \int_\R \left[(u_1(\tau,x)-1)^+ + (u_2(\tau,x)-1)^+\right] \,dx \le
    \int_0^\tau\int_\R r(u_1,u_2)\,dxdt, 
  \end{equation*}
  with
  \begin{equation*}
    r(u_1,u_2)=S(u_1,u_2)(H(u_1-1)-H(u_2-1)).
  \end{equation*}
  We have that
  \begin{align*}
    r(u_1,u_2)&=
    \begin{cases}
      0 & u_1 < 1\ \text{and}\ u_2< 1,\\
      0 & u_1 > 1\ \text{and}\ u_2> 1,\\
      \left[\left(v_2(u_2)-v_1(u_1)\right)^+ u_1 -
    \left(v_2(u_2)-v_1(u_1)\right)^- u_2 \right]  & u_2 \le 1 < u_1,\\
     -\left[\left(v_2(u_2)-v_1(u_1)\right)^+ u_1 -
    \left(v_2(u_2)-v_1(u_1)\right)^- u_2 \right]  & u_1 \le 1 < u_2,
    \end{cases}\\
    &\le 0,
  \end{align*}
  if $u_1$ and $u_2$ are non-negative. 
\end{proof}

\begin{remark}
  There are also other invariant regions for this equation.  If 
  \begin{equation*}
    v_2\left(u_{2,0}(x)\right)\ge v_1\left(u_{1,0}(x)\right),
  \end{equation*}
  then 
  \begin{equation*}
    v_2\left(u_{2}(t,x)\right)\ge v_1\left(u_{1}(t,x)\right)
  \end{equation*}
  for $t>0$.
\end{remark}
\subsection{An example}\label{subsec:example}
We finish our discussion of the two-lane case by exhibiting an
example. The velocities on the two roads are
\begin{equation}\label{eq:twolanev}
  v_1(u)=1.5(1-u)\ \text{and}\ v_2(u)=2.5(1-u),
\end{equation}
and the initial data
\begin{equation}\label{eq:twolaneinit}
  u_{1,0}(x)=u_{2,0}(x)=\sin^2(\pi x/2).
\end{equation}
Of course, we do not have entropy solutions in closed form, so instead
we use a numerical approximation generated by the Engquist--Osher
scheme with $800$ grid points in the interval $[0,2]$.
Figure~\ref{fig:1} shows the computed solution at $t=0.375$, $t=0.75$,
$t=1.125$ and $t=1.5$. For comparison, we have also included the
single lane model with the (average of $v_1$ and $v_2$) speed
$v(u)=2(1-u)$.  We see that there is the expected change of lane to
the faster lane, and that a shock builds up in the fast lane to the left of 
the shock in the slow lane.
\begin{figure}[h]
  \centering
    \begin{tabular}{lr}
    \includegraphics[width=0.4\linewidth]{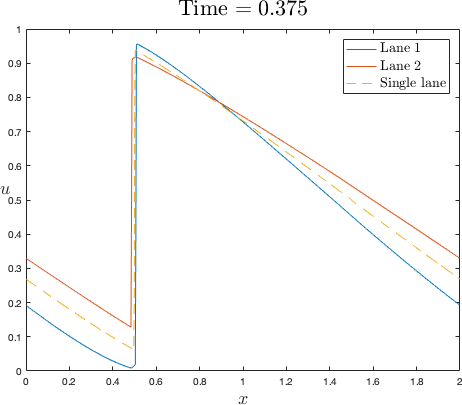}
    &\includegraphics[width=0.4\linewidth]{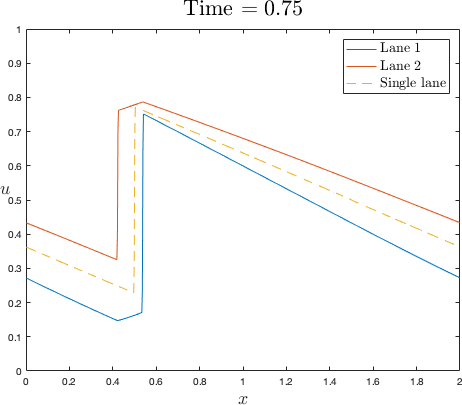}\\
    \includegraphics[width=0.4\linewidth]{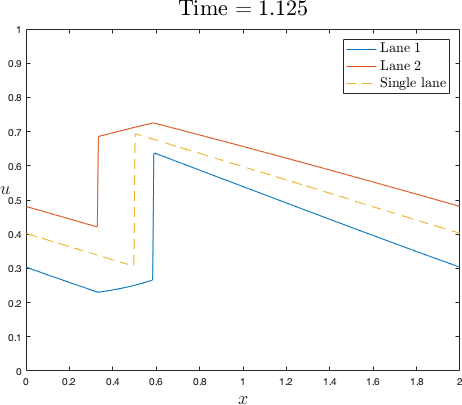}
    &\includegraphics[width=0.4\linewidth]{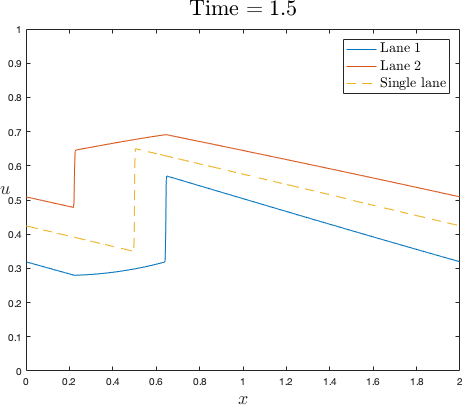}
  \end{tabular}
  \caption{The computed solutions of \eqref{eq:ligningen} with $v_1$
    and $v_2$ given by \eqref{eq:twolanev} and initial data given
    by \eqref{eq:twolaneinit}.}
  \label{fig:1}
\end{figure}
\section{Multilane model} \label{sec:cont2a}
The model \eqref{eq:ligningen} can be generalized to an arbitrary number of lanes.  Consider a road with $N$ lanes. Traffic is unidirectional and dense. 
Each lane has its specific velocity function $v_i$ depending only on the density in that lane, thus $v_i=v_i(u_i)$, where $u_i$ is the density in lane $i$. 

Assume that drivers prefer to drive in the faster lane, and this
tendency increases with the velocity difference with adjacent lanes. 
Thus the
flow \emph{from} lane $i$ \emph{to} lane $i+1$ equals
\begin{equation*}
  S_i(u_i,u_{i+1}) = \left[\left(v_{i+1}(u_{i+1})-v_i(u_i)\right)^+ u_i -
    \left(v_{i+1}(u_{i+1})-v_i(u_i)\right)^- u_{i+1} \right],
\end{equation*}
where we have scaled time such that the constant of proportionality is one.
We then get, in the analogous manner to the derivation
of \eqref{eq:ligningen}, that
\begin{equation}
  \label{eq:carbalance}
  \partial_t u_i + \partial_x \left(u_iv_i(u_i)\right) =
  S_{i-1}(u_{i-1},u_i) - S_i(u_i,u_{i+1}), \qquad i=1,\ldots,N,
\end{equation}
coupled with the boundary conditions
\begin{equation}
  \label{eq:noflow}
  S_0(u_0,u_1)=S_N(u_N,u_{N+1})=0.
\end{equation}

\begin{definition} \label{def:N_def_solution} Let $v_i=v_i(u_i)$ be
 Lipschitz continuous functions, and assume
  that $u_{i,0}\in L^1(\R)\cap L^\infty(\R)$, for $i=1,\dots,N$. We say
  that $u_{i}\in C([0,\infty); L^1(\R))$ is a weak solution of
  \eqref{eq:carbalance} with initial data $u_{i,0}$ if
  \begin{multline*}
    \int_0^\infty \int_\R\big( u_{i}\test_t+u_1
    v_i(u_i)\test_x+
    \left(S_{i}(u_{i},u_{i+1})-S_{i-1}(u_{i-1},u_{i+1})\right)\test
    \big)\,  dxdt\\
    + \int_\R u_{i,0}\test|_{t=0}\, dx=0, \quad i=1,\dots,N,
  \end{multline*}
  for all compactly supported test functions $\test\in
  C^\infty([0,\infty)\times \R)$.

  It is an entropy solution if
  \begin{multline}\label{eq:kruzN}
    \int_0^\infty \int_\R\big(  \eta(u_{i})\test_t+q_i(u_i)\test_x\big)\, dxdt+ \int_\R \eta(u_{i,0})\test|_{t=0}\, dx\\
    \ge\int_0^\infty \int_\R
    \eta'(u_i)
    \left((S_{i}(u_{i},u_{i+1})-S_{i-1}(u_{i-1},u_{i})\right)\test \,
    dxdt, \quad i=1,\dots,N,
  \end{multline}
  for all convex functions $\eta$,  and for all
  non-negative test functions $\test\in
  C^\infty_0([0,\infty)\times \R)$. Here $q_i$ is defined by
  $q_i'(u)=\eta'(u)f_i'(u)$ with $f_i(u)=u v_i(u)$.
\end{definition}
The wellposedness of the system of equations~\eqref{eq:carbalance} is
ensured by the following general theorem from \cite{HKR_weakly}, see also \cite{HKLR_split}.
\begin{theorem}[\cite{HKLR_split}, Theorem~3.13] \label{thm:existence}
  Assume that $v_i$ and $u_{i,0}$ are as in
  Definition~\ref{def:N_def_solution}. The there exists a unique
  entropy solution $u=\seq{u_i}_{i=1}^N$. Furthermore, if
  $\baru=\seq{\baru_i}_{i=1}^N$ is another entropy solution with
  initial data $\seq{\baru_{i,0}}_{i=1}^N$, then
  \begin{multline}
    \label{eq:generalstability}
    \sum_{i=1}^N \norm{u_i(t,\dott)-\baru_i(t,\dott)}_{L^1(\R)} \\
    \le 
    \sqrt{N} \exp\left(2N\sup_i\lipnorm{S_i} t\right) \sum_{i=1}^N
    \norm{u_{i,0}-\baru_{i,0}}_{L^1(\R)}. 
  \end{multline}
\end{theorem}
A fundamental property of hyperbolic conservation law is the $L^1$ contractivity of solutions in the sense that the spatial $L^1$-norm of the difference between two
entropy solutions at a specific time does not increase in time.  This
property is in general lost for weakly coupled systems, or for scalar
conservation laws with a source. The general bound
\eqref{eq:generalstability} does not imply $L^1$
contractivity. However, for the system \eqref{eq:carbalance}, the
special form of the source yields $L^1$ contractivity for the whole solution, as the next theorem shows.
\begin{theorem}\label{thm:L1contr}
  Consider two entropy solutions $u=\seq{u_i}_{i=1}^N$ and
  $\bar{u}=\seq{\baru_i}_{i=1}^N$ of \eqref{eq:carbalance} with
  initial data $u_0=\seq{u_{i,0}}$ and $\baru_0=\seq{\baru_i}$,
  respectively.  Then we have
  \begin{equation}
    \label{eq:L1disc}
    \sum_{i=1}^N \int_\R \abs{u_i(x,t)-\baru_i(x,t)}\,dx \le
    \sum_{i=1}^N \int_\R \abs{u_{i,0}(x)-\baru_{i,0}(x)}\,dx.
  \end{equation}
\end{theorem}
\begin{proof}
  By using Kru\v{z}kov's doubling of variables technique we get
  \begin{multline*}
    \partial_t \abs{u_i-\baru_i} + \partial_x
    \left[\sgn{u_i-\baru_i}(f_i(u_i)-f_i(\baru_i))\right] \\ \qquad
    \le -\sgn{u_i-\baru_i}
    \left[S_i(u_i,u_{i+1})-S_i(\baru_i,\baru_{i+1}) -
      S_{i-1}(u_{i-1},u_i)-S_{i-1}(\baru_{i-1},\baru_i)\right]
  \end{multline*}
  in $\mathcal{D}'$. Subtracting the equation for $u_i$ and adding
  the equation for $\baru_i$ we arrive at
  \begin{multline*}
    \partial_t (u_i-\baru_i)^+ + \partial_x\left[
      H(u_i-\baru_i)\left(f_i(u_i)-f_i(\baru_i)\right)\right]\\
    \qquad \le -H(u_i-\baru_i)
    \left[S_i(u_i,u_{i+1})-S_i(\baru_i,\baru_{i+1}) -
      S_{i-1}(u_{i-1},u_i)-S_{i-1}(\baru_{i-1},\baru_i)\right],
  \end{multline*}
  in $\mathcal{D}'$.
  Choosing $\test\approx \indicM{[0,\tau]}$ we infer that
  \begin{equation}
    \label{eq:plussstab}
    \begin{aligned}
      & \int_\R \left(u_i(x,\tau)-\baru_i(x,\tau)\right)^+\,dx \\
      &\quad\le \int_\R \left(u_i(x,0)-\baru_i(x,0)\right)^+\,dx
      +\int_0^\tau \int_\R
      H\left(u_i-\baru_i\right)\\
      &\qquad \times
      \left[S_{i-1}(u_{i-1},u_i)-S_{i-1}(\baru_{i-1},\baru_i)
        -\left(S_i(u_i,u_{i+1})-S_i(\baru_{i},\baru_{i+1})\right)\right]\,dxdt.
    \end{aligned}
  \end{equation}
  Recall that 
  \begin{equation*}
    S_i(a,b) = \left(v_{i+1}(b)-v_i(a)\right)^+ a -
    \left(v_{i+1}(b)-v_i(a)\right)^- b.
  \end{equation*}
  Now
  \begin{align*}
    \frac{\partial S_i}{\partial a} &=
    \left(v_{i+1}(b)-v_i(a)\right)^+ - H\left(v_{i+1}(b)-v_i(a)\right)
    v_i'(a)(a+b)\ge 0,\\
    \intertext{and}
    \frac{\partial S_i}{\partial b} &=
    H\left(v_{i+1}(b)-v_i(a)\right)v_{i+1}'(b)a -
    \left(v_{i+1}(b)-v_i(a)\right)^- \\
    &\qquad +
    H\left(-\left(v_{i+1}(b)-v_i(a)\right)\right) v_{i+1}'(b)a \le 0.
  \end{align*}
  So if $u_i>\baru_i$,
  \begin{align*}
    &S_{i-1}(u_{i-1},u_i)-S_{i-1}(\baru_{i-1},\baru_i)
    -\left(S_i(u_i,u_{i+1})-S_i(\baru_{i},\baru_{i+1})\right)\\
    &\qquad\le S_{i-1}(u_{i-1},\baru_i) - S_{i-1}(\baru_{i-1},\baru_i)
    -
    (S_i(u_i,u_{i+1})-S_i(u_i,\baru_{i+1}))\\
    &\qquad\le c
    \max\seq{u_{i-1},\baru_{i-1}}\left(u_{i-1}-\baru_{i-1}\right)^+ +
    c
    \max\seq{u_{i+1},\baru_{i+1}}\left(u_{i+1}-\baru_{i+1}\right)^+\\
    &\qquad\le c\left[\left(u_{i-1}-\baru_{i-1}\right)^+ +
      \left(u_{i+1}-\baru_{i+1}\right)^+\right],
  \end{align*}
  where $0<c<\abs{v_i'}$. Therefore
  \begin{align*}
    \sum_{i=1}^N
    &H\left(u_i-\baru_i\right)\left[S_{i-1}(u_{i-1},u_i)-S_{i-1}(\baru_{i-1},\baru_i)
      -\left(S_i(u_i,u_{i+1})-S_i(\baru_{i},\baru_{i+1})\right)\right]\\
    &\qquad \le 2c \sum_{i=1}^N \left(u_i-\baru_i\right)^+.
  \end{align*}
  Define
  \begin{equation*}
    \Theta(t) = \int_\R \sum_{i=1}^N
    \left(u_i(x,t)-\baru_i(x,t)\right)^+ \,dx,
  \end{equation*}
  then \eqref{eq:plussstab} and the above inequality imply that
  \begin{equation*}
    \Theta(T) \le \Theta(0) + 2c\int_0^t \Theta(t)\,dt.
  \end{equation*}
  Gronwall's inequality then implies that
  \begin{equation*}
    \Theta(T)\le \Theta(0)e^{2cT}.
  \end{equation*}
  Thus if $\Theta(0)=0$, i.e., $u_{i,0}(x)\le \baru_{i,0}(x) \;
  \text{a.e.}\; x$, then $\Theta(T)=0$ for $T>0$, i.e., $u_{i}(x,T)\le
  \baru_i(x,T) \; \text{a.e.}\; x$.

  By the Crandall--Tartar lemma \cite[Lemma 2.13]{HoldenRisebro}, this
  implies $L^1$ contractivity, i.e., if $u$ and $\baru$ are entropy
  solutions to \eqref{eq:carbalance} with initial data $u_0$ and
  $\baru_0$, then \eqref{eq:L1disc} holds for $t>0$.  
\end{proof}
One way to
enforce the boundary conditions \eqref{eq:noflow}, is to define
$ u_0(x,t)=u_1(t,x)$, $v_0(u)=v_1(u)$, $u_{N+1}(x,t)=u_N(x,t)$ and
$v_{N+1}(u)=v_N(u)$. Henceforth we will use this convention.
\begin{corollary} Consider two solutions $u_i$ and $\baru_i$ of \eqref{eq:carbalance} with initial data $u_{i,0}$ and $\baru_{i,0}$, respectively, 
in the sense of Definition \ref{def:N_def_solution}.  Then we have
\begin{equation}
  \label{eq:BVy}
  \sum_{i=1}^{N-1}
  \norm{u_{i+1}(\dott,t)-u_{i}(\dott,t)}_{L^1(\R)} 
  \le
  \sum_{i=1}^{N-1} \norm{u_{i+1,0}-u_{i,0}}_{L^1(\R)}.
\end{equation}
Furthermore, we have
\begin{equation}
  \label{eq:BVx}
  \sum_{i=1}^N \abs{u_i(\dott,t)}_{BV(\R)} \le \sum_{i=1}^N \abs{u_{i,0}}_{BV(\R)}.
\end{equation}
In addition
\begin{equation}\label{eq:L1time}
  \sum_{i=1}^N \norm{u_i(\dott,t+h)-u_i(\dott,t)}_{L^1(\R)}\le
  \sum_{i=1}^N \norm{u_i(\dott,h)-u_i(\dott,0)}_{L^1(\R)}.
\end{equation}
\end{corollary}
\begin{proof}
Setting $\baru_{i,0}=u_{i+1,0}$ for $i=1,\ldots,N$ yields \eqref{eq:BVy}.
Similarly, defining $\baru_{i,0}(x)=u_{i,0}(x+h)$, using
\eqref{eq:L1disc}, and sending $h$ to zero gives  \eqref{eq:BVx}.
To obtain time continuity we define $\baru_{i.0}(x)=u_i(x,h)$, to
get \eqref{eq:L1time}. 
\end{proof}
We also note the following useful estimates.
Define $f_i(u)=u v_i(u)$ and $\DM a_i=a_i-a_{i-1}$, divide by $h$ and let
$h\downarrow 0$ to find that
\begin{equation}\label{eq:altbv}
  \sum_{i=1}^N \norm{f_i(u_i)_x-\Delta^-
    S_{i}(u_i, u_{i+1})}_{L^1(\R)} \le
  \sum_{i=1}^N \norm{f_i(u_{i,0})_x-\Delta^- S_{i}(u_{i,0},u_{i+1,0})}_{L^1(\R)}.
\end{equation}
If we assume that the quantity on the left is bounded by $C$, then we
get
\begin{equation}
  \label{eq:timecont}
  \sum_{i=1}^N \norm{u_i(\dott,t+h)-u_i(\dott,t)}_{L^1(\R)}\le Ch.
\end{equation}
Furthermore, we have the useful observation
\begin{equation}\label{eq:useful}
  \sum_{i=1}^N \norm{\Delta^- S_{i}((u_{i},u_{i+1})(\dott,t))}_{L^1(\R)} \le C +
  \sum_{i=1}^N\abs{f_i(u_{i,0})}_{BV(\R)}.
\end{equation}
\subsection{An example}\label{subsec:manylane_example}
We also here include an example. For $i=1,\ldots,8$ we set
$u_{i,0}(x)=\sin^2(\pi x/2)$, and define 
\begin{equation}\label{eq:v8}
  v_i(u)=k_i(1-u), \quad k_i= \frac{13}{12} + \frac{i-1}{4}, \qquad i=1,\ldots,8.
\end{equation}
Also in this case the depicted solutions were calculated with the
Engquist--Osher scheme with $800$ grid points in the interval $[0,2]$. 
Figure~\ref{fig:2} shows the computed solutions at $t=0.375$,
$t=0.75$, $t=1.125$ and $t=1.5$. We see the expected change of lanes to
the faster lanes, and that a shock builds up in the faster lanes to the left of 
the  slower lanes.
\begin{figure}[h]
  \centering
  \begin{tabular}{lr}
    \includegraphics[width=0.4\linewidth]{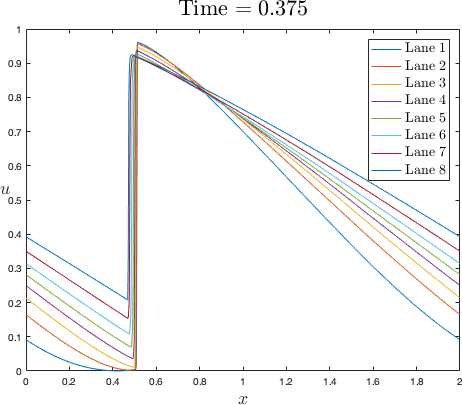}
    &\includegraphics[width=0.4\linewidth]{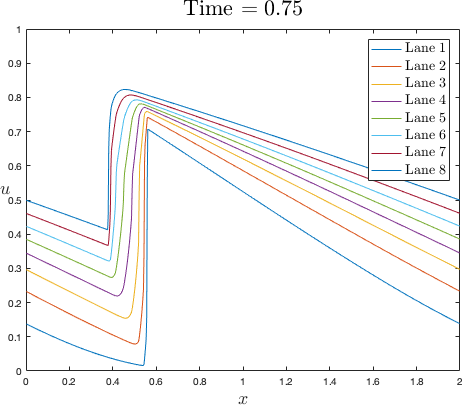}\\
    \includegraphics[width=0.4\linewidth]{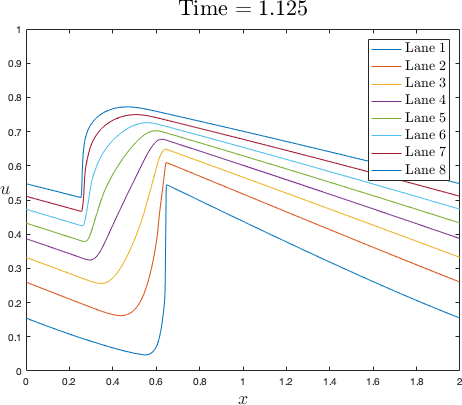}
    &\includegraphics[width=0.4\linewidth]{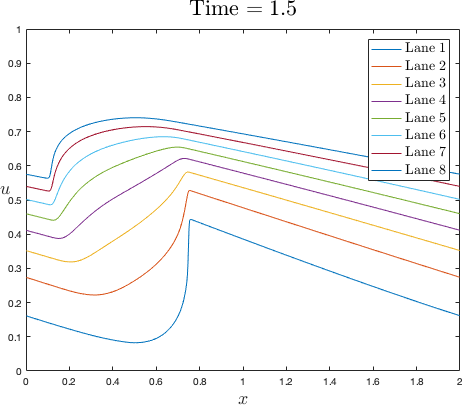}
  \end{tabular}
  \caption{The solution of \eqref{eq:carbalance} with $N=8$, and $v_i$
    given by \eqref{eq:v8}. Upper left: $t=0.375$; upper right:
    $t=0.75$; lower left: $t=1.125$; lower right $t=1.5$.} 
  \label{fig:2}
\end{figure}

\section{Infinitely many lanes --- the continuum limit}
\label{sec:inftylane}
It is natural, at least mathematically, to consider the case where the
lanes increase in number while at the same time get closer. Our aim in
this section is therefore to investigate limit as $N\to \infty$ in the
system in the previous section. 

To this end we let (the number of lanes) $N$ be a positive integer and
set $\Dy=1/N$. Let $y_i=(i-1/2)\Dy$ for $i=1,\ldots,N$. We shall also use the
``divided difference'' notation 
\begin{equation*}
  D^{\pm} a_i = \pm \frac{a_{i\pm 1}-a_i}{\Dy}.
\end{equation*}
For simplicity, we restrict our presentation to the case where
$v_i(u)=-k(y_i) g(u)$ where $g$ is a differentiable function with
$g'(u)>0$, $g(0)=-1$ and $g(1)=0$. Define
$f(u)=-ug(u)$. Throughout we will use the notation $f_i=f(u_i)$,
$g_i=g(u_i)$ and $k_i=k(y_i)$. Now we reintroduce the scaling constant
$K$ in \eqref{eq:diffdef}, and set $K=\kappa/\Dy^2$.  For the convenience of the reader
we set $\kappa=1$.
Thus, for $i=1,\ldots,N$, $u_i$ is the unique entropy (in the sense of
Definition~\ref{def:N_def_solution}) solution of the balance equation
\begin{equation}
  \label{eq:carbalance_new}
  \partial_t u_i + k_i \partial_x f(u_i) = \frac{1}{\Dy^2}\left[
  S_{i-1}(u_{i-1},u_i)-S_{i}(u_i,u_{i+1})\right],
\end{equation}
 with the boundary conditions 
\begin{equation*}
  u_0=u_1,\ u_{N+1}=u_N,\ k_0=k_1\ \text{and}\ k_{N+1}=k_N.
\end{equation*}
It is also useful to define the function $u_{\Dy}(t,x,y)$ by
\begin{equation} \label{eq:piecewiseconstant}
  u_{\Dy}(t,x,y)=
  \begin{cases}
    u_i(t,x), &\text{if $y\in [y_{j-1/2},y_{j+1/2})$,
      $i=1,\ldots,N-1$,}\\
    u_N(t,x) &\text{if $y\in [y_{N-1/2},1]$.}
    \end{cases}
\end{equation}
We shall investigate whether the family $\seq{u_{\Dy}}_{\Dy=1/N}$,
$N\in\N$ is compact and characterize the limit $\lim_{\Dy\to 0}
u_{\Dy}$. To this end we must show a number of estimates.

The right-hand
side of \eqref{eq:carbalance_new} equals
\begin{align}
  \frac{1}{\Dy^2}(S_{i-1}-S_{i}) &= - u_i\Dp\Dm V_i +
  \Dp u_i\left(\Dp V_i\right)^- - \Dm u_i \left(\Dm V_i\right)^+\notag\\
  &=u_i\Dp\Dm(k_ig_i) + \underbrace{\Dp u_i\left(\Dp k_ig_i\right)^+ - \Dm
  u_i\left(\Dm k_ig_i\right)^-}_{b_i}\notag\\
  &=\Dp\left(u_i\Dm(k_ig_i)\right) - \Dp u_i\Dm k_ig_i \notag \\
  &\qquad + \Dp u_i\left(\Dp k_ig_i\right)^+ - \Dm
  u_i\left(\Dm k_ig_i\right)^-\notag\\
  &=\Dp\left(u_i\Dm(k_ig_i)\right) + \DP u_i\left(\Dp k_ig_i\right)^-
  -
  \Dm u_i\left(\Dm k_ig_i\right)^-\notag\\
  &=\Dp\left(u_i\Dm(k_ig_i)\right) + \Dy\, \Dm\left((\Dp u_i)\left(\Dp
      k_ig_i\right)^-\right).\label{eq:source}
\end{align}
Thus \eqref{eq:carbalance_new} reads
\begin{equation}
  \label{eq:carbalance2}
  \partial_t u_i + k_i \partial_x f(u_i) =
  \Dp\left(u_i\Dm(k_ig_i)\right) + \Dy\, \Dm\left((\Dp u_i)\left(\Dp
      k_ig_i\right)^-\right) 
\end{equation}
for $i=1,\ldots,N$, and we have the boundary values
\begin{equation}
  \label{eq:carboundary}
  \Dm(k_1g_1)=\Dp(k_N g_N)=0.
\end{equation}
\begin{remark*}
  Observe that the above term $b_i$ is an upwind discretization of the
  transport term corresponding to $a u_y$, with $a=(kg)_y$.
\end{remark*}
 
Similarly to \eqref{eq:source}, we also get the expression
\begin{equation}\label{eq:altsource}
  \frac{1}{\Dy^2}\left(S_{i-1}-S_{i}\right) =
  \Dm\left(u_i \Dp\left(k_i g_i\right)\right) + \Dy \Dm\left((\Dp u_i) \left(\Dp k_ig_i\right)^+\right).
\end{equation}
Recall \eqref{eq:kruzN} with $\eta(u)=u^2/2$ and $\test$ an approximation to $\indicM{[0,T]}$.  That gives
\begin{align*}
  \frac12 \int_\R (u_i(x,T))^2 \,dx &\le \frac12 \int_\R
  (u_{i,0}(x))^2\,dx \\
  &\quad + \int_{\Pi_T} \left(u_i\Dp\left(u_i\Dm(k_ig_i)\right) + \Dy\,
  u_i \,\Dm\left(\Dp u_i\left(\Dp k_ig_i\right)^-\right)\right)\,dxdt,
\end{align*}
where $\Pi_T=[0,T]\times \R$.
We can sum this for $i=1,\ldots,N$, multiply with $\Dy$ and do a
summation by parts to get
\begin{multline} 
  \frac12 \Dy \sum_{i=1}^N  \int_\R (u_i(x,T))^2 \,dx \\
  + \Dy \sum_{i=1}^N \intom u_i\Dm(k_ig_i) \Dm u_i\,dxdt +
    \Dy^2\sum_{i=1}^N \intom \left(\Dp k_ig_i\right)^-
  \left(\Dp u_i\right)^2 \,dxdt \\
   \le\frac12 \Dy\sum_{i=1}^N \int_\R (u_{i,0}(x))^2\,dx. \label{eq:entro}
 \end{multline} 
It will be useful to lower bound the last two terms on the left-hand
side. 

Recall first that
\begin{equation} \label{eq:ubounds}
  0\le u_i \le 1,\ \abs{\Dp k_i}\le C \ \text{and}\ \Dy\sum_{i=1}^N \intom \abs{\Dp
    u_i} \,dxdt \le C,
\end{equation}
for some constant $C$ independent of $\Dy$. Using this and the fact that
$\max_{u\in[0,1]}\abs{g(u)}$ is bounded, as well as
\begin{equation} \label{eq:estD}
\Dy \abs{D^\pm u_i}\le C,
\end{equation}
we have that
\begin{align}
  \Dy^2\sum_{i=1}^N \intom \abs{g_i\Dp k_i}\left(\Dp
    u_i\right)^2 \,dxdt &\le C \Dy^2\sum_{i=1}^N \intom
  \left(\Dp
    u_i\right)^2 \,dxdt \notag\\
  &\le C\Dy\sum_{i=1}^N \intom  \abs{\Dp u_i} \,dxdt \le C.  \label{eq:est1A}
\end{align}
Furthermore, note that the same argument yields
\begin{equation}
  \Dy \sum_{i=1}^N \intom \abs{u_i g_{i-1}(\Dm k_i)(\Dm
    u_i)}\,dxdt \le C \Dy \sum_{i=1}^N \intom\abs{\Dm u_i}\,dxdt \le C.\label{eq:est1B}
\end{equation}

Observe that
\begin{equation*}
  \Dp k_ig_i = k_{i+1}\Dp g_i + g_i\Dp k_i,
\end{equation*}
and then use the inequality $(a+b)^- \ge a^- - \abs{b}$. Thus, since
$g'>0$,
\begin{align*}
  \left(\Dp k_ig_i\right)^- \left(\Dp u_i\right)^2 &\ge
  k_{i+1}\left(\Dp g_i\right)^-(\Dp u_i)^2 -
  \abs{g_i\Dp k_i}\left(\Dp u_i\right)^2\\
  &\ge c\left((\Dp u_i)^-\right)^3 - \abs{g_i\Dp k_i}\left(\Dp
    u_i\right)^2,
\end{align*}
where $0<c\le \min_i k_{i} \min_u g'(u)$. 
Similarly,
\begin{equation*}
  \Dm(k_ig_i)=k_i \Dm g_i + g_{i-1}\Dm k_i,
\end{equation*}
and therefore
\begin{equation*}
  u_i\Dm(k_ig_i) (\Dm u_i) \ge k_iu_i(\Dm g_i)(\Dm u_i) - \abs{u_i g_{i-1}(\Dm
    k_i)(\Dm u_i)}.
\end{equation*}
Note that due to the monotonicity of $g$ we have for some $\tilde u$ between $u_i$ and $u_{1-1}$, 
\begin{equation*}
k_iu_i\Dm g_i\Dm u_i = k_iu_i g'(\tilde u)(\Dm u_i)^2 \ge c u_i (\Dm u_i)^2 \ge 0.
\end{equation*}
We can now estimate the last two terms of the left-hand side of \eqref{eq:entro} from below. More 
precisely,
\begin{align*}  
& c\Dy^2\sum_{i=1}^N \intom \left((\Dp u_i)^-\right)^3 \,dxdt -
 \Dy^2\sum_{i=1}^N \intom \abs{g_i\Dp k_i}\left(\Dp u_i\right)^2 \,dxdt \\
 &\qquad +c\Dy\sum_{i=1}^N \intom u_i\left(\Dm u_i)\right)^2 \,dxdt
  -\Dy\sum_{i=1}^N \intom \abs{u_i g_{i-1}\Dm k_i\Dm u_i} \,dxdt \\
 &\qquad\le  \Dy^2\sum_{i=1}^N \intom \left(\Dp k_ig_i\right)^-
  \left(\Dp u_i\right)^2 \,dxdt +\Dy \sum_{i=1}^N \intom u_i\Dm(k_ig_i) \Dm u_i\,dxdt \\
  &\qquad\le \frac12 \Dy \sum_{i=1}^N  \int_\R (u_i(x,T))^2 \,dx+\Dy^2\sum_{i=1}^N \intom \left(\Dp k_ig_i\right)^-
  \left(\DP u_i\right)^2 \,dxdt \\
&\qquad\qquad\qquad
  +\Dy \sum_{i=1}^N \intom u_i\Dm(k_ig_i) \Dm u_i\,dxdt \\
  &\qquad\le \frac12 \Dy\sum_{i=1}^N \int_\R (u_{i,0}(x))^2\,dx,
\end{align*}
which we can rewrite as
\begin{align*}  
&c\Dy^2\sum_{i=1}^N \intom \left((\Dp u_i)^-\right)^3 \,dxdt +c\Dy\sum_{i=1}^N \intom u_i\left(\Dm u_i)\right)^2 \,dxdt \\
&\qquad\le
\frac12 \Dy\sum_{i=1}^N \int_\R (u_{i,0}(x))^2\,dx  +\Dy^2\sum_{i=1}^N \intom \abs{g_i\Dp k_i}\left(\Dp u_i\right)^2 \,dxdt \\
&\qquad\quad
 +\Dy\sum_{i=1}^N \intom \abs{u_i g_{i-1}(\Dm k_i)(\Dm u_i)} \,dxdt \\
& \qquad\quad +\Dy\sum_{i=1}^N \intom \abs{u_i g_{i-1}(\Dm k_i)(\Dm u_i)} \,dxdt  \\
&\qquad \le C,
\end{align*}
using \eqref{eq:est1A} and \eqref{eq:est1B}.

This implies that
\begin{equation}
  \label{eq:L3bound}
  \Dy^2 \sum_{i=1}^N \intom \left((\Dp u_i)^-\right)^3 \,dxdt
  \le C,
\end{equation}
and
\begin{equation}
  \label{eq:L2typebnd}
  \Dy \sum_{i=1}^N \intom u_i\left(\Dm u_i\right)^2 \,dxdt \le C.
\end{equation}
Observe that by \eqref{eq:estD}, \eqref{eq:L3bound} follows from \eqref{eq:L2typebnd}, viz.
\begin{equation*}
 \Dy^2 \sum_{i=1}^N \intom \left((\Dp u_i)^-\right)^3 \,dxdt
  \le   \Dy \sum_{i=1}^N \intom u_i\left(\Dm u_i\right)^2 \,dxdt \le  C.
\end{equation*}
By the same procedure, starting with \eqref{eq:carbalance2} but using
the alternate form \eqref{eq:altsource} of the right-hand side, we
arrive at the bounds
\begin{equation}
  \label{eq:altL3bound}
  \Dy^2 \sum_{i=1}^N \intom \left((\Dp u_i)^+\right)^3 \,dxdt
  \le C,
\end{equation}
and
\begin{equation}
  \label{eq:altL2typebnd}
  \Dy \sum_{i=1}^N \intom u_i\left(\Dp u_i\right)^2 \,dxdt \le C.
\end{equation}
Combining the two bounds \eqref{eq:L3bound} and \eqref{eq:altL3bound}
we get
\begin{equation}
  \label{eq:finalL3bound}
  \Dy^2 \sum_{i=1}^N\intom \abs{\Dp u_i}^3 \,dxdt \le C.
\end{equation}
In a similar manner, we find
\begin{equation}
  \label{eq:finalL3bound_minus}
  \Dy^2 \sum_{i=1}^N\intom \abs{\Dm u_i}^3 \,dxdt \le C.
\end{equation}

The other two bounds, \eqref{eq:L2typebnd} and \eqref{eq:altL2typebnd}
can be used for a continuity estimate. Write
$u_{i-1/2}=(u_i+u_{i-1})/2$ and compute for $\ell\ge m$
\begin{align*}
  \frac12\abs{u_\ell^2 - u_m^2} &= \frac{\Dy}{2} \Bigl|\sum_{i=m+1}^\ell \Dm u_i^2\Bigr|\\
  &=\Dy \Bigl|\sum_{i=m+1}^\ell u_{i-1/2} \Dm u_i\Bigr| \\
  &\le \Bigl(\Dy \sum_{i=m+1}^\ell u_{i-1/2}\Bigr)^{1/2}
  \Bigl(\Dy\sum_{i=m+1}^\ell u_{i-1/2}\left(\Dm
    u_i\right)^2\Bigr)^{1/2}\\
  &\le \sqrt{\Dy(\ell-m)} 
  \Bigl(\frac{\Dy}2 \sum_{i=1}^N u_i\left((\Dm u_i)^2 + (\Dp
    u_i)^2\right)\Bigr)^{1/2}.
\end{align*}
Squaring and integrating over $[0,T]\times \R$ gives
\begin{equation}
  \label{eq:L2cont}
  \intom \left(u_\ell^2 - u_m^2\right)^2 \,dxdt \le C(\ell-m)\Dy, \quad \ell\ge m.
\end{equation}

By direct computations we have that
\begin{equation*}
  \frac12 \Dm u_i^2=u_{i-1/2} \Dm u_i 
  =u_i \Dm u_i - \frac{\Dy}{2}\left(\Dm u_i\right)^2,
\end{equation*}
which gives
\begin{align*}
   \left(u_i\Dm u_i\right)^2  &= \frac14 \left(\Dm u_i^2\right)^2 +
   \Dy u_i \left(\Dm u_i\right)^3 - \frac{\Dy^2}{4}\left(\Dm
     u_i\right)^4\\
   &\le \frac14 \left(\Dm u_i^2\right)^2 +
   \Dy \abs{\Dm u_i}^3.
\end{align*}
Multiplying with $\Dy$ summing over $i$ and integrating in $x$, $t$,
gives the bound, using \eqref{eq:L2cont} with $m=i-1$, $\ell=i$ and \eqref{eq:finalL3bound_minus},
\begin{equation}
  \label{eq:L2derbnd}
  \Dy \sum_{i=1}^N \intom \left(u_i \Dm u_i\right)^2\,dxdt
  \le C.
\end{equation}
Note that this also follows from \eqref{eq:L2typebnd}, using that $u_i\in[0,1]$.

\subsubsection*{Convergence}
We assume that $u_0\colon\R\times[0,1]\to \R$ is such that $0\le
u_0(x,y)\le 1$ and that $u_0\in L^1\cap BV$.
Now we assume that the initial data $u_{i,0}$ are such that there is a
function $u_0(x,y)$ such that 
\begin{equation}  \label{eq:init_N}
  u_{i,0}(x)=\frac{1}{\Dy} \int_{y_{i-1/2}}^{y_{i+1/2}} u_0(x,y)\,dy
  \in L^1(\R) \ \text{for $i=1,\ldots,N$},
  \end{equation}
where $\Dy=1/N$ and $y_{i-1/2}=(i-1)\Dy$. Furthermore $0\le u_0(x,y)\le 1$. Since
$u_0\in BV(\R\times [0,1])$,
\begin{equation*}
  \Dy \sum_{i=1}^N \abs{u_{i,0}}_{BV(\R)} + \Dy \int_\R \sum_{i=1}^N
  \abs{D^{\pm}u_{i,0}} \,dx \le C
\end{equation*}
for some constant $C$ which is independent of $\Dy$. For convenience,
we have set $u_{0,0}=u_{0,1}$ and $u_{0,N+1}=u_{0,N}$.

We assume that $k\in C^1([0,1])$ is given, such that
$k'(0)=k'(1)=0$, and $k(y)>0$ for $y\in [0,1]$. Define 
  $k_i = k(y_i)$.
Let $u_i(t,x)$ be the entropy solutions to
\eqref{eq:carbalance2} with the boundary conditions
\begin{equation*}
  \Dm k_1 = \Dp k_N = 0, \quad \Dm u_1 = \Dp u_N = 0,
\end{equation*}
which actually is a special case of \eqref{eq:carboundary}. Then we
define 
\begin{equation*}
  u_{\Dy}(t,x,y)=u_i(t,x) \quad \text{for $y\in [y_{i-1/2},y_{i+1/2})$},
\end{equation*}
for $i=1,\ldots,N-1$ and $u_{\Dy}(t,x,y)=u_N(t,x)$ if $y\in
[y_{N-1/2},1]$.
We have that $0\le u_{\Dy}(t,x,y)\le 1$,
$\norm{u_{\Dy}(t,\dott,\dott)}_{L^1(\R\times[0,1])} =
\norm{u_0}_{L^1(\R\times[0,1])}$, and, using the bounds \eqref{eq:BVy}
and \eqref{eq:BVx}, $\norm{u_{\Dy}(t,\dott,\dott)}_{BV(\R\times[0,1])}\le
C$, where $C$ is independent of $\Dy$. Furthermore, using
\eqref{eq:L2cont},
\begin{equation*}
  \norm{u_{\Dy}(t,\dott,\dott)-u_{\Dy}(s,\dott,\dott)}_{L^1(\R\times[0,1])}
  \le C\abs{t-s},
\end{equation*}
where $C$ is independent of $\Dy$. This is sufficient to conclude that
there is a function $u \in C([0,\infty);L^1(\R\times [0,1]))$ and a
sequence $\seq{\Dy_j}_{j=0}^\infty$, $\Dy_j\to 0$ as $j\to \infty$, such that 
\begin{equation*}
  u=\lim_{j\to \infty} u_{\Dy_j} \quad \text{in $ C([0,\infty);L^1(\R\times [0,1]))$.} 
\end{equation*}
Furthermore, we have that $\Dm u_{\Dy_j} \rightharpoonup u_y$,
therefore $u_{\Dy_j}\Dm u_{\Dy_j} \rightharpoonup u u_y$. The bound
\eqref{eq:L2derbnd} ensures that $uu_y\in L^2([0,T]\times
\R\times[0,1])$.  

\begin{definition}\label{def:infinity}
  Set $\Omega=\R\times[0,1]$ and $\Omega_T=[0,T]\times \Omega$.
  Let $k=k(y)$ be as above, in particular
  $k'(0)=k'(1)=0$. We say that $u\in C([0,\infty);L^1(\Omega))$, such that
  $uu_y\in L^2(\Omega_T)$, is a
  weak solution to
  \begin{equation*}
    \begin{cases}
      u_t + kf(u)_x + (k' f(u))_y = \left(k u g_y\right)_y \quad t>0,
      \quad (x,y) \in  \R\times (0,1),\\
      g(u)_y = 0 \quad x\in \R, \quad y=0, \quad y=1,\\
      u(0,x,y)=u_0(x,y)\quad (x,y) \in \R\times (0,1),
    \end{cases}
  \end{equation*}
  if for all test functions $\test\in C^\infty_0(\Omega_T)$,
  \begin{equation*}
    \begin{aligned}
      \intomega\big( u\test_t + k f(u)&\test_x + k'f(u)\test_y\big)
      \,dydxdt
      = \intom \int_0^1 k u g'(u) u_y \test_y \,dydxdt\\
      &+\intomeg u(T,x,y) \test(T,x,y)\,dxdy - \intomeg u_0
      \test(0,x,y)\,dxdy.
    \end{aligned}
  \end{equation*}
\end{definition}

The aim is now to show that the limit $u$ is a weak solution in the
above sense. Since $u_i$ is a weak solution of
\eqref{eq:carbalance2}, we have
\begin{align}
  \intom \big(u_i\test_t + k_i f_i &\test_x - \Dp\left((\Dm k_i)
    f_i\right) \test\big)\,dxdt \label{eq:term1}\\
  &=
  -\intom \Dp\left(k_{i-1}u_i\Dm g_i\right) \test \,dxdt \label{eq:term2}\\
  &\quad -
  \Dy \intom \Dm\left(\Dp
    u_i\left(\Dp k_ig_i\right)^-\right) \test \,dxdt\label{eq:term3}\\
  &\quad +\int_\R  u_i(T,x) \test(T,x)\,dxdy - \int_\R
       u_{i,0} \test(0,x)\,dx \label{eq:term4}
\end{align}
for $i=1,\ldots,N$. We use  $\test=\test^i$ where
\begin{equation*}
  \test^i(t,x)=\frac{1}{\Dy}\int_{y_{i-1/2}}^{y_{i+1/2}} \test(t,x,y)\,dy
\end{equation*}
for a suitable test function $\test$. Next we multiply with $\Dy$ and
sum over $i=1,\ldots,N$ and do a summation by parts on the terms which
have $D^{\pm}(\cdots)$. This will give us
the weak formulation for $u_{\Dy}$. For simplicity we assume that
$\Dy=\Dy_j$, so that the whole sequence converges. Term by term we get
\begin{align*}
  \Dy\sum_{i=1}^N \text{\eqref{eq:term1}}&=
   \Dy\sum_{i=1}^N \intom \left(u_i\test^i_{t} + k_i f_i \test^i_x
   + (\Dm k_i)
    f_i\Dm \test^i \right)\,dxdt\\
    &\longrightarrow\; 
    \intom\Bigl( \int_0^1 u\test_t + kf\test_x + k' f\test_y
    \,dy\Bigr)\,dxdt 
\end{align*}
as $\Dy\to 0$. 

Turning to \eqref{eq:term2}, we have that 
\begin{equation*}
  \Dm g_i = g'(\tilde u_{i-1/2}) \Dm u_i = g'(u_i)\Dm u_i +
  g''(\xi_{i-1/2})(u_i - \tilde u_{i-1/2})\Dm u_i, 
\end{equation*}
where $\tilde u_{i-1/2}$ is between $u_i$ and $u_{i-1}$ and $\xi_{i-1/2}$ is
between $u_i$ and $\tilde u_{i-1/2}$. Therefore
\begin{align}
  \Dy \sum_{i=1}^N \text{\eqref{eq:term2}} &= \Dy \sum_{i=1}^N
  \intom k_{i-1} u_i g'(u_i) \Dm u_i \Dm\test^i \,dxdt \notag\\
  &\quad + \Dy\sum_{i=1}^N \intom  k_{i-1}u_i g''(\xi_{i-1/2})\left(u_{i}-\tilde u_{i-1/2}\right)
  \Dm u_i \Dm\test^i\,dxdt.\label{eq:smallterm}
\end{align}
The last term here vanishes as $\Dy\to 0$ since
\begin{align*}
  \abs{\text{\eqref{eq:smallterm}}} &\le C\Dy^2 \sum_{i=1}^N
    \intom u_i \left(\Dm u_i\right)^2 \abs{\Dm\test^i}
    \,dxdt\\
    &\le C\Dy,
  \end{align*}
where we used \eqref{eq:L2typebnd}. Hence
\begin{equation*}
  \Dy \sum_{i=1}^N \text{\eqref{eq:term2}} \;\longrightarrow\;
  \intomega
  k u g'(u) u_y \test_y \,dydxdt,
\end{equation*}
as $\Dy\to 0$.

Now for \eqref{eq:term3}, we have
\begin{align*}
  \Dy \Bigl|\sum_{i=1}^N \text{\eqref{eq:term3}}\Bigr| 
  &\le \Dy^2
  \sum_{i=1}^N \intom \abs{\DP u_i} \left(k_{i-1} \abs{\Dp
      g_i} + \abs{g_i}\, \abs{\Dp k_i}\right) \abs{\Dp \test^i}\,dxdt \\
  &\le C\Dy\Bigl( \Dy \sum_{i=1}^N \intom \left(\Dp u_i\right)^2 \,dxdt
  +
  \Dy\sum_{i=1}^N \intom \abs{\Dp u_i}\,dxdt\Bigr)
  \\
  &\le C\Dy\Bigl(\Dy\sum_{i=1}^N \intom \abs{\Dp u_i}\,dxdt
  \Bigr)^{1/2}\Bigl(\Dy\sum_{i=1}^N \intom \abs{\Dp
    u_i}^3\,dxdt\Bigr)^{1/2} \\
    &\qquad\qquad\qquad\qquad\qquad\qquad\qquad\qquad\qquad\qquad+ C\Dy
  \\
  &\le C\Dy\left(\frac{1}{\sqrt{\Dy}}+1\right),
\end{align*}
using \eqref{eq:ubounds}, \eqref{eq:finalL3bound}, and interpolation
between $L^1$ and $L^3$. Thus $\Dy \bigl|\sum_{i=1}^N\text{\eqref{eq:term3}}\bigr| \rightarrow
0$ as $\Dy\to 0$.

It is straightforward to show that 
\begin{equation*}
  \Dy \sum_{i=1}^N \text{\eqref{eq:term4}} \;\longrightarrow\;\intomeg
  u(T,x,y)\test(T,x,y) \,dydx - \intomeg
  u_0(x,y)\test(0,x,y) \,dydx.
\end{equation*}
Hence, the limit $u$ is a weak solution.

We can sum up the result of our arguments in the following theorem.
\begin{theorem} \label{thm:infinite_theorem}
  Let $k\in C^2([0,1])$ such that $k'(0)=k'(1)=0$, and $k(y)>0$ for
  all $y\in [0,1]$, and assume that $g=g(u)$ is a strictly increasing
  differentiable function such that $g(0)=-1$ and $g(1)=0$.
 
  Assume that  $u_0\in L^1(\Omega)\cap BV(\Omega)$ and let $u_{\Dy}$ be
  defined in \eqref{eq:piecewiseconstant} where $u_i$ solves
  \eqref{eq:carbalance_new} for $i=1,\ldots,N$.

  Then there exists a
  sequence $N_j\to \infty$ and correspondingly $\Dy_j=1/N_j\to 0$ such
  that the sequence of solutions $\seq{u_{\Dy_j}}_{j=1}^\infty$
  has a limit, i.e.,
  \begin{equation*}
    u=\lim_{j\to \infty} u_{\Dy_j} \quad \text{in $ C([0,\infty);L^1(\Omega))$.} 
  \end{equation*}
  The limit $u$ is a weak solution according to
  Definition~\ref{def:infinity}. 

  We also have the regularity estimate
  \begin{equation}
    \norm{u^2(y_1) - u^2(y_2)}_{L^2([0,T]\times\R)}^2  \le
    C\abs{y_1-y_2}, \quad y_1,y_2\in[0,1]. 
  \end{equation}
\end{theorem}

\subsection{An example}\label{subsec:cont_example}
To illustrate the continuum limit, we have tested the ``same'' initial
value problem as in Section~\ref{subsec:example} and
Section~\ref{subsec:manylane_example}. The relevant data are
\begin{equation*}
  u_0(x,y)=\sin^2(\pi x/2), \quad x\in \R, \ y\in (0,1),
\end{equation*}
and 
\begin{equation}
  \label{eq:vmanylane}
  k(y)=1+2y, \ y\in (0,1),\quad v(y,u)=k(y)(1-u).
\end{equation}
We have used $\Dy=1/60$ (i.e., 60 lanes) and solved
\eqref{eq:carbalance_new} using the Engquist--Osher scheme with $800$
grid points in the interval $[0,2]$.
Figure~\ref{fig:3} shows the computed density $u$ at four different
times. It is illuminating to compare this figure with
Figures~\ref{fig:1} and \ref{fig:2}. 
\begin{figure}[h]
  \begin{tabular}{lr}
    \includegraphics[width=0.5\linewidth]{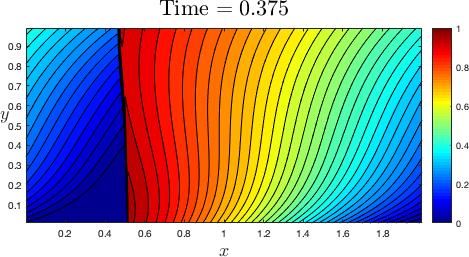}
    &\includegraphics[width=0.5\linewidth]{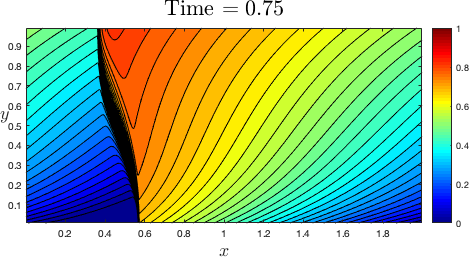}\\
    \includegraphics[width=0.5\linewidth]{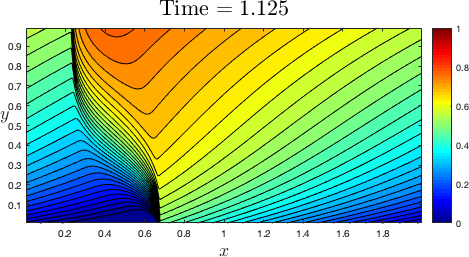}
    &\includegraphics[width=0.5\linewidth]{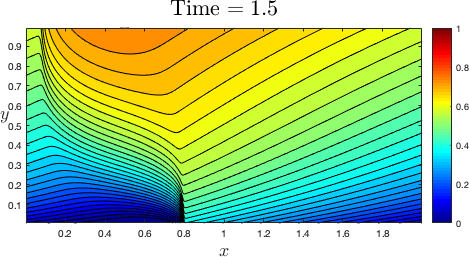}
  \end{tabular}
  \caption{The solution of \eqref{eq:carbalance_new} with $N=60$, and
    $v(y,u)$ given by \eqref{eq:vmanylane}. Upper left: $t=0.375$;
    upper right: $t=0.75$; lower left: $t=1.125$; lower right
    $t=1.5$.}
  \label{fig:3}
\end{figure}



\end{document}